\documentclass[11pt,reqno,a4paper]{article}
\usepackage{geometry}
\usepackage{amsmath,amsthm,amstext,amscd,amssymb,euscript,url}
\usepackage{mathrsfs}
\usepackage{calrsfs}
\usepackage{epsfig}
\usepackage[inline]{enumitem}
\usepackage{mathtools}
\usepackage[absolute,overlay]{textpos}
\usepackage{subfig}
\usepackage[utf8]{inputenc}
    \usepackage[dutch,english]{babel}
    \usepackage{lipsum}
    \usepackage{dsfont}
    \usepackage{graphicx}
	\usepackage{caption}
	\usepackage[utf8]{inputenc}
	\usepackage{float}
	\usepackage{enumitem}
    \usepackage{xfrac}
    \usepackage{hyperref}
    \usepackage{upgreek}

\usepackage{color}

\usepackage{appendix}
    \numberwithin{equation}{section}
    
    \newcommand{\p}{\mathbb{P}}
    \newcommand{\Om}{\Omega}

\newcommand{\R}{\mathbb R}
\newcommand{\N}{\mathbb N}

\newcommand{\E}{\mathbb E}

\renewcommand{\phi}{\varphi}

\newcommand{\pee}{\ensuremath{\mathbb{P}}}

\newcommand{\hE}{\widehat{\mathbb{E}}}
\newcommand{\hP}{\widehat{\mathbb{P}}}

\newcommand{\dd}{\mathop{}\!\mathrm{d}}

\def\1{{\mathchoice {\rm 1\mskip-4mu l} {\rm 1\mskip-4mu l}
{\rm 1\mskip-4.5mu l} {\rm 1\mskip-5mu l}}}

\newtheorem{theorem}{{\small T}{\scriptsize HEOREM}}[section]
\newtheorem{corollary}{{\bf{\small C}{\scriptsize OROLLARY}}}[section]
\newtheorem{proposition}{{\bf{\small P}{\scriptsize ROPOSITION}}}[section]
\newtheorem{lemma}{{\bf{\small L}{\scriptsize EMMA}}}[section]

 \theoremstyle{definition}
\newtheorem{remark}{{\bf{\small R}{\scriptsize EMARK}}}[section]
\newtheorem{definition}{{\bf{\small D}{\scriptsize EFINITION}}}[section]
\newtheorem{induction}{{\bf{\small I}{\scriptsize NDUCTIVE HYPOTHESIS}}}[section]

\renewenvironment{proof}[1]
{\noindent{{\bf{\small{ P}{\scriptsize ROOF}}}.}\hspace{0.1cm} #1} {$\;\qed$\newline}

\newcommand{\beq}{\begin{eqnarray}}
\newcommand{\eeq}{\end{eqnarray}}

\newcommand{\ba}{\begin{align*}}
\newcommand{\ea}{\end{align*}}

\newcommand{\be}{\begin{equation}}
\newcommand{\ee}{\end{equation}}

\newcommand{\bl}{\begin{lemma}}
\newcommand{\el}{\end{lemma}}

\newcommand{\br}{\begin{remark}}
\newcommand{\er}{\end{remark}}

\newcommand{\bt}{\begin{theorem}}
\newcommand{\et}{\end{theorem}}

\newcommand{\bd}{\begin{definition}}
\newcommand{\ed}{\end{definition}}

\newcommand{\bind}{\begin{induction}}
\newcommand{\eind}{\end{induction}}

\newcommand{\bp}{\begin{proposition}}
\newcommand{\ep}{\end{proposition}}

\newcommand{\bc}{\begin{corollary}}
\newcommand{\ec}{\end{corollary}}

\newcommand{\bpr}{\begin{proof}}
\newcommand{\epr}{\end{proof}}

\newcommand{\bi}{\begin{itemize}}
\newcommand{\ei}{\end{itemize}}

\newcommand{\ben}{\begin{enumerate}}
\newcommand{\een}{\end{enumerate}}

\newcommand{\caC}{{\mathscr C}}

\newcommand{\caD}{{\EuScript D}}

\newcommand{\caI}{{\mathcal I}}

\newcommand{\caT}{{\mathcal T}}

\newcommand{\RTP}{\text{\normalfont RTP}}

\newcommand{\SEP}{\text{\normalfont SEP}}

\newcommand{\nn}{\nonumber}

\newmuskip\pFqmuskip

\newcommand\pFq[6][8]{%
	\begingroup 
	\pFqmuskip=#1mu\relax
	\mathcode`\,=\string"8000
	\begingroup\lccode`\~=`\,
	\lowercase{\endgroup\let~}\pFqcomma
	{}_{#2}F_{#3}{\left[\genfrac..{0pt}{}{#4}{#5};#6\right]}%
	\endgroup
}
\newcommand{\pFqcomma}{\mskip\pFqmuskip}

\newtheorem{asu}{Assumption}
\newcommand{\basu}{\begin{asu}}
\newcommand{\easu}{\end{asu}}

\title{Ergodic theory of multi-layer interacting particle systems}
\author{Frank Redig\footnote{Delft Institute of Applied Mathematics, TU Delft, Delft, The Netherlands. f.h.j.redig@tudelft.nl}\ , Hidde van Wiechen\footnote{Delft Institute of Applied Mathematics, TU Delft, Delft, The Netherlands. h.vanwiechen@tudelft.nl} }
\date{\today}

\begin{document}

\maketitle
\begin{center}
    \section*{Abstract}
    \end{center}
    
\begin{changemargin}{10mm}{10mm}
We consider a class of multi-layer interacting particle systems and characterize the set of ergodic probability measures with finite moments. The main technical tool is duality combined with successful coupling.\\

{\bf Keywords:} Interacting particle systems, duality, multi-layer random walks, coupling.
\end{changemargin}

\section{Introduction}
In this paper we consider a class of multi-layer particle systems with duality properties.

The study of multi-layer random walks and interacting particle systems is motivated by the study of active particles, where the motion is
determined by an internal degree of freedom (determining the direction of motion) and a random component which models the influence of collisions with surrounding particles. Other terms for such random motions are e.g. persistent random walk and run-and-tumble motion.

When the internal degree of freedom takes a finite number of values and evolves autonomously as a finite state space continuous Markov chain, then
one can view the motion as a random walk on a multi-layer system, where the layers are indexed by the internal states.
There are various studies of the asymptotic properties of such random walks including law of large numbers (asymptotic speed), invariance principle (central limit behavior), and large deviations.  See e.g. \cite{maj, maj2, maj3, gros, gink} and references therein.

In this paper, we study a system of particles performing multi-layer random walks which possibly have interaction of
inclusion or exclusion type (see Section \ref{Models} below for a precise description of the models). We characterize the set of invariant probability measures with finite moments. More precisely we prove that
under an appropriate condition of moment growth, the only ergodic invariant probability measures are homogeneous product probability measures, indexed by the first moment (particle density). In the case of independent particles, these probability measures are product Poisson measures, in the case of interacting systems they are
products of binomial (exclusion) resp.\ negative binomial (inclusion) distributions.

The characterization of the invariant probability measures of multi-layer exclusion processes has been obtained recently in \cite{Amir}.
The study of the hydrodynamic limit of a system of active particles has been studied in \cite{erignoux}, and for a two-layer system 
with duality in \cite{simone}.

The important ingredient in our setting here is duality combined with the existence of a successful coupling for the dual process. 
This road is followed for the symmetric exclusion process in \cite{ligg} chapter 8.
Duality allows to characterize the invariant probability measures via the characterization of bounded harmonic functions of the dual process, which is
a countable state space Markov chain. If this Markov chain admits a successful coupling, then the bounded harmonic functions are constants, indexed by the number of dual particles.
The proof of the existence of a successful coupling is a combination of coupling the finite state space internal state process and the Ornstein coupling of random walks. These two ingredients are sufficient in the non-interacting case. In the interacting case, we use the approach of \cite{kuoch} which consists of ``spreading out'' the particles combined with Ornstein coupling of symmetric random walks.

The rest of our paper is organized as follows.  In Section \ref{2} we provide the general setup of  multi-layer particle systems, after which we define the three types of processes (exclusion, inclusion and independent walkers) we will study in this paper. Afterwards, we study the duality properties and invariant probability measures of these processes.

In Section \ref{3} we provide a characterization of the ergodic invariant probability measures in a slightly more general setting, where the only assumptions are duality with polynomial duality functions and the existence of a successful coupling. This unifies and generalizes earlier results from chapter 8 of the book of Liggett \cite{ligg} and Kuoch \cite{kuoch}.

Section \ref{4} is devoted to the proof of a successful coupling for the models under consideration. For independent particles this amounts to generalize the Ornstein coupling to the multi-layer setting. For interacting particles, it amounts to generalize the approach of Liggett for the exclusion process \cite{ligg} and Kuoch for the inclusion process \cite{kuoch}.

\section{Models and their duality properties}\label{2}
\subsection{Models: definitions}\label{Models}
In this paper we will look at models of configurations where the coordinates of individual particles are of the form $(x,\sigma)$, with $x \in \mathbb{Z}^d$ the \emph{position} of the particle and $\sigma \in S$ the \emph{internal state}, where $S$ is some finite set. We will denote the single particle state space as $V:=\mathbb{Z}^d\times S$, which we will think of as $|S|$ layers of $\mathbb{Z}^d$. For this reason, we will also refer to $\sigma \in S$ as the layer on which a particle at $(x,\sigma)$ resides.

We consider a configuration process $\{\eta(t):t \geq 0\}$ on a state space $\Om_s$  that will be defined later. 
The generator of the process is of the following type,
\begin{equation}\label{generator}
\mathscr{L}_sf(\eta)= \sum_{v,w\in V} p(v,w) \eta_{v}(\alpha + s\eta_{w}) \nabla_{v,w}f(\eta).
\end{equation}
Here $\eta_v$ is equal to the number of particles at site $v\in V$ and, if we denote $\eta^{v\to w}$ as the configuration $\eta$ where a single particle has moved from $v$ to $w$ (if possible), we have
\begin{equation*}
    \nabla_{v,w}f(\eta) = f(\eta^{v\to w}) - f(\eta).
\end{equation*}
 The value of $s\in\{-1,0,1\}$ in \eqref{generator} determines the type of the process we consider (exclusion, inclusion or independent particles). For this reason, the parameter $s$ also determines the state space $\Om_s$ we consider and the single particle transition rates $p(v,w)$ and constants $\alpha \in \mathbb{R}_+$ we allow. Below we will define the process for each possible value of $s$. 
 
 The main characteristic of our multi-layer particle systems is that the transition rates are determined by the layer on which a particle resides. Therefore, for every $\sigma \in S$ we consider a nearest neighbor symmetric random walk on $\mathbb{Z}^d$ with translation invariant transition rates. We denote by $\pi_\sigma(x)$ the corresponding rate to jump from $z$ to $z+x$. Note that  $\pi_\sigma(x) > 0$ if and only if $|x|=1$ and that $\pi_\sigma(x) = \pi_\sigma(-x)$. Furthermore, we let $\{c(\sigma,\sigma'): \sigma,\sigma' \in S\}$ be transition rates on the set of layers $S$ which we will assume to be symmetric and irreducible. Then we define the following processes:
 
 \begin{enumerate}
     \item \textbf{Symmetric exclusion process ($s=-1$).} Every site contains at most $\alpha \in \mathbb{N}$ particles and  jumps to sites where there are already many other particles are  less likely. 
     The state space of the multi-layer SEP is given by $\Om_{-1} = \{0,1,...,\alpha\}^V$,  and  the single particle transition rates   we will study for this model are of the following form,
 \begin{equation}\label{transition interacting}
     p\big((x,\sigma),(y,\sigma')\big) = \pi_\sigma(y-x)\delta_{\sigma,\sigma'} + c(\sigma,\sigma')\delta_{x,y},
 \end{equation}
 with $ \delta_\cdot$ the Kronecker delta.
     \item \textbf{Symmetric inclusion process ($s=1$).} In contrast to the exclusion process, this process actually encourages jumps to sites where other particles already reside. The state space of the multi-layer SIP is given by $\Om_1 = \mathbb{N}^V$, and the transition rates are again given by \eqref{transition interacting}. Furthermore, we allow for any $\alpha>0$.
     \item \textbf{Independent particles ($s=0$).}  For this paper, our model for independent particles will be  the run-and-tumble particle process (\RTP). A run-and-tumble particle is a particle with the following dynamics.
\begin{itemize}
    \item[] \emph{Random walk jumps.} With rate $\kappa$, a particle at $(x,\sigma)$ performs a nearest neighbor symmetric random walk jump on $\mathbb{Z}^d$ according to the transition rates $\pi_\sigma(\cdot)$, i.e., $(x,\sigma) \to (x + y,\sigma)$ with rate $\kappa \pi_{\sigma}(y)$.
    \item[] \emph{Active jumps.} With rate $\lambda$, a particle at $(x,\sigma)$ performs an active jump in the direction determined by the internal state $\sigma$, i.e., there exists a function $v:S \to \mathbb{Z}^d$ such that $(x,\sigma) \to (x + v(\sigma),\sigma)$ with rate $\lambda$.
    \item[] \emph{Internal state jumps.} A particle changes its internal state according to the transition rates  $\{c(\sigma, \sigma'):\sigma,\sigma'\in S\}$.
\end{itemize}
The state space of this process is $\Om_0 = \mathbb{N}^V$, and from the dynamics we conclude that the single particle transition rates are of the following form,
\begin{equation}\label{transition rtp}
    p\big((x,\sigma),(y,\sigma')\big) = \kappa \pi_{\sigma}(y-x)\delta_{\sigma,\sigma'} + \lambda \delta_{\sigma,\sigma'}\delta_{ y,x+v(\sigma)} + c(\sigma,\sigma')\delta_{x,y}.
\end{equation}
In this case, any choice of $\alpha>0$ is possible. However, without loss of generality we put $\alpha=1$. 
Furthermore,  in the special case where $\kappa=0$ we will assume that $\lambda>0$ and that the range of $v$ spans the whole of $\mathbb{Z}^d$, i.e.,
\begin{equation}\label{irreducible internal}
    \text{vct}\{\mathscr{R}(v)\} = \mathbb{Z}^d.
\end{equation}
This condition is crucial in order to construct the successful coupling cf. \eqref{pong} below.
 \end{enumerate}
 
 \begin{remark}
 Notice that for the interacting models we only allow for symmetric transitions on every layer. This is because for asymmetric transition rates we only have duality when the particles are independent. 
 \end{remark}

\subsection{Duality}
We will state and prove  duality results for the processes we just defined. Recall the following definition of duality of Markov processes.
\begin{definition}
Let $\{\eta(t):t\geq0\}$ and $\{\xi(t):t\geq0\}$ be two Markov processes on the state spaces $\Om$ and $\Om'$ respectively, and let $D:\Om'\times \Om \to \mathbb{R}$ be a measurable function. We say that $\{\eta(t):t\geq0\}$ and $\{\xi(t):t\geq0\}$ are \emph{dual} to one another, with respect to $D$, if
\begin{equation}\label{duality def}
    \E_\eta \left[D(\xi,\eta(t))\right] = \widehat{\E}_\xi \left[D(\xi(t),\eta)\right].
\end{equation}
Here $\E_\eta$ denotes the expectation in $\{\eta(t):t \geq 0\}$ starting from $\eta$,  $\hE_\xi$ the expectation in the dual process $\{\xi(t): t\geq 0\}$ starting from $\xi$, and we assume that both sides are bounded. We then call $D$ the \emph{duality function}.
\end{definition}
\subsubsection{Duality results}\label{dual processes}
Let $|\xi| := \sum_x \xi_x$ denote the number of particles in $\xi$ and let $\Om_{s,f}:=\{\xi \in \Om : |\xi|<\infty \}$ be the subspace of $\Om_s$  consisting of only those configurations with a finite number of particles. In the following theorem we will give duality results of the processes defined in section \ref{Models} with duality functions $D_s:\Om_{s,f} \times \Om_s \to \mathbb{R}$ of the following form:
\begin{equation}\label{duality function ds}
    D_s(\xi, \eta) = \prod_{v\in V} d_s(\xi_v, \eta_v). 
\end{equation}
The proof of the  first two statements  can be found in for example \cite[Theorem 4.1]{Giar}. For the third statement, we will make use of another duality result, with the so-called \emph{associated deterministic system}, that is introduced in \cite{carinci}. 
\begin{theorem}\label{duality}
\begin{enumerate}
\item If $s=-1$ then the process generated by $\mathscr{L}_{-1}$ is self-dual with duality function
\begin{align*}
    D_{-1}(\xi,\eta) &= \prod_{v \in V} \frac{\eta_v!}{(\eta_v-\xi_v)!} \cdot \frac{(\alpha - \xi_v)!}{\alpha!}\cdot  I(\xi_v \leq \eta_v),
\end{align*}
where $I(\cdot)$ denotes the characteristic function.
\item If $s=1$, then the process generated by $\mathscr{L}_1$ is self-dual with duality function
\begin{align*}
    D_1(\xi,\eta)&=\prod_{v \in V} \frac{\eta_v!}{(\eta_v-\xi_v)!} \cdot\frac{\Gamma(\alpha)}{\Gamma(\alpha +\xi_v)}\cdot  I(\xi_v \leq \eta_v).
\end{align*}
\item \label{duality 3} If $s=0$ then the process generated by $\mathscr{L}_0$, with transition rates $p(v,w)$ given by \eqref{transition rtp}, is dual to its time-reversed process, i.e., the RTP process with single particle transition rates
\be
\widehat{p}\big((x,\sigma),(y,\sigma')\big) = \kappa \pi_{\sigma}(y-x)\delta_{\sigma,\sigma'} + \lambda \delta_{\sigma,\sigma'}\delta_{ y,x-v(\sigma)} + c(\sigma,\sigma')\delta_{x=y},
\ee
and with the same parameter $\alpha>0$. We denote this as the $\widehat{\text{RTP}}$ process. The corresponding duality function is given by 
\begin{equation*}
D_0(\xi,\eta) = \prod_{v \in V} \frac{\eta_v!}{(\eta_v-\xi_v)!}\cdot  I(\xi_v \leq \eta_v).
\end{equation*}
\end{enumerate}
\end{theorem}
\begin{remark}\label{single particle duality}
If $\xi=\delta_v$ is the configuration containing a single particle at $v\in V$ and no particles elsewhere, then $D(\xi,\eta) = c_{\alpha, s}\eta_v$, with $c_{\alpha,s}$ a positive constant depending on the model ($s\in \{-1,0,1\}$) and the constant $\alpha$. As a consequence we have that 
\begin{equation}\label{bongo}
\begin{split}
\E_\eta[\eta_v(t)] = \frac{1}{c_{\alpha,s}} \E_\eta [D(\delta_v, \eta(t))]= \frac{1}{c_{\alpha,s}} \widehat{\E}_v[D(\delta_{v(t)},\eta)] =\widehat{\E}_v[\eta_{v(t)}],
\end{split}
\end{equation}
where $\widehat{\E}_v$ denotes the expectation of the dual process starting from $\delta_v$.
\end{remark}

\subsubsection{Proof of duality for the RTP process}
Let $\{v(t):t\geq0\}$ be the random path of a single particle in $V$ performing the RTP dynamics starting from $v(0) = v$. The deterministic system we will consider is the following: for a function $f:V \to \mathbb{R}$, define 
\begin{equation*}
    f_t(v) := \sum_{w\in V} p_t(v,w) f(w) = \E\left[ f\big(v(t)\big)\right],
\end{equation*} where $p_t(v,w)$ is the transition kernel of a single RTP particle. In other words, the process $\{f_t : t\geq 0\}$ follows the Kolmogorov backwards equation of the RTP process. We now have the following duality result:
\begin{proposition} Let $f:V\to\mathbb{R}$ be such that $f(v) \neq 1$ for only a finite number of $v\in V$.
For the deterministic processes  $\{f_t : t\geq 0\}$ and the process $\{\eta(t):t\geq0\}$ generated by $\mathscr{L}_0$, it holds that
\begin{equation}\label{deterministic duality}
     \E_\eta\left[ \prod_{v \in V} f(v)^{\eta_v(t)}\right] =\prod_{v\in V} f_t(v)^{\eta_v},
\end{equation}
i.e., the two processes are dual to one another with duality function 
\begin{equation*}
    \mathscr{D}(f,\eta) = \prod_{v \in V} f(v)^{\eta_v}.
\end{equation*}
\end{proposition}
 The proof of this result is straightforward and only relies on the fact that the particles in the RTP process move independently.

\begin{proof}
Define $\{v_i(t) : i \in I, t\geq 0\}$ as the paths of the particles in the configuration $\eta(t)$ with $I$ an arbitrary set of labels, i.e., $\eta_v(t) = \sum_{i \in I}I(v_i(t) = v)$ for all $v \in V$ and $t\geq 0$. We then have that 
\begin{equation*}
    \E\left[ \prod_{v \in V} f(v)^{\eta_v(t)}\right] = \E\left[ \prod_{i} f(v_i(t))\right] =\prod_{i} \E\left[f(v_i(t))\right] = \prod_i f_t(v_i) = \prod_{v\in V} f_t(v)^{\eta_v},
\end{equation*}
where in the second step we used the independence of particles.
\end{proof}

We will now prove item \ref{duality 3} of Theorem \ref{duality} using the duality result in \eqref{deterministic duality}.

\begin{proof}
Let $\{\xi(t):t\geq0\}$ denote an $\widehat{\text{RTP}}$ process and, for a given $\eta$, define the sequence of finite configurations $(\eta^{(N)})_{N\in\mathbb{N}}$ as 
\begin{equation*}
    \eta^{(N)}_{(x,\sigma)} := \begin{cases}
    \eta_{(x,\sigma)}\ \ \ \ \ &\text{if $x \in [-N,N]$,}\\
    0&\text{else.}
    \end{cases}
\end{equation*} We will first prove the result for the dual process starting from $n$ particles at position $v_i \in V$, i.e., $\xi = n\cdot \delta_{v_i}$, with $\delta_{v_i}$ the configuration with a single particle at $v_i$, and by replacing the starting configuration $\eta$ by $\eta^{(N)}$. By taking the $n$-th order derivative with respect to $f(v_i)$ on the left-hand side of \eqref{deterministic duality} and afterwards setting $f\equiv 1$, we find that 
\begin{align}\label{dual n1}
\nn\left.\frac{\partial^n}{\partial f(v_i)^n}\E_{\eta^{(N)}} \left[ \prod_{v \in V} f(v)^{\eta^{(N)}_v(t)}\right]\right|_{f\equiv 1} &= \left.\E_{\eta^{(N)}}\left[\frac{\partial^n}{\partial f(v_i)^n}  \prod_{v \in V} f(v)^{\eta^{(N)}_v(t)}\right]\right|_{f\equiv 1}\\
&= \E_{\eta^{(N)}} \left[\frac{\eta_{v_i}(t)!}{(\eta_{v_i}(t)-n)!}\cdot  I\big(n \leq \eta_{v_i}(t)\big)\right].
\end{align}
Here we  were able to interchange the derivatives and the expectation using dominated convergence. Note that the right-hand side is equal to  $\E_{\eta^{(N)}}\left[D_0(\xi, \eta^{(N)}(t))\right]$.

Applying the same operations as in \eqref{dual n1} to the right hand side of \eqref{deterministic duality}, we obtain
\begin{align*}
&\left. \frac{\partial^n}{\partial f(v_i)^n} \prod_{v \in V}\left(\sum_{w \in V} p_t(v,w) f(w)\right)^{\eta^{(N)}_v} \right|_{f\equiv 1}\\
&\ \ \ =\left.\sum_{m=1}^n \sum_{\substack{v^{(1)}, ..., v^{(m)} \in V \\ v^{(i)}\neq v^{(j)}}} \sum_{k_1 + ... + k_m = n} \binom{n}{k_1, ..., k_m}\prod_{j=1}^m \frac{\partial^{k_j}}{\partial f(v_i)^{k_j}} \left(\sum_{w \in V} p_t(v^{(j)},w)f(w)\right)^{\eta^{(N)}_{v^{(j)}}}\right|_{f\equiv1}\\
&\ \ \ =\sum_{m=1}^n \sum_{\substack{v^{(1)}, ..., v^{(m)} \in V \\ v^{(i)}\neq v^{(j)}}} \sum_{k_1 + ... + k_m = n} \binom{n}{k_1, ..., k_m}\prod_{j=1}^m d_0\left(\eta^{(N)}_{v^{(j)}}, k_j\right)p_t\left(v^{(j)}, v_i\right)^{k_j} \\
&\ \ \ =\sum_{m=1}^n \sum_{\substack{v^{(1)}, ..., v^{(m)} \in V \\ v^{(i)}\neq v^{(j)}}} \sum_{k_1 + ... + k_m = n} \binom{n}{k_1, ..., k_m}\prod_{j=1}^m d_0\left(\eta^{(N)}_{v^{(j)}}, k_j\right)\widehat{p}_t\left(v_i,v^{(j)}\right)^{k_j}.
\end{align*}
Here $\widehat{p}_t(w,v)$ is the transition kernel of a single $\widehat{\text{RTP}}$ particle, and we have used that $p_t(v,w) = \widehat{p}_t(w,v)$ for all $v,w \in V$. Notice that the last line in the above formula is the expected value of $D_0(\xi(t),\eta^{(N)})$, i.e.,
\begin{equation}\label{dual n2}
\left. \frac{\partial^n}{\partial f(v_i)^n} \prod_{v \in V}\left(\sum_{w \in V} p_t(v,w) f(w)\right)^{\eta^{(N)}_v} \right|_{f\equiv 1} = \widehat{\E}_\xi \big[D_0(\xi(t),\eta^{(N)})\big].
\end{equation}
Combining \eqref{deterministic duality}, \eqref{dual n1} and \eqref{dual n2}, we find that 
$$\E_{\eta^{(N)}}\big[D_0(\xi,\eta^{(N)}(t))\big] = \widehat{\E}_\xi\big[D_0(\xi(t),\eta^{(N)})\big].$$
The claim now follows from monotone convergence as $N\to\infty$. 

If we consider  any finite configuration of particles $\xi \in \Om_{0,f}$, i.e., $\xi = \sum_{i=1}^n \delta_{v_i}$ for some $n \in \mathbb{N}$ and $v_i \in V$, then the duality result can be found by taking the derivative with respect to each $f(v_i)$ on both the left-hand side and right-hand side of the equation \eqref{deterministic duality}. 
\end{proof}

\subsection{Invariant probability measures}
\begin{proposition}\label{invariant}
For the processes defined in section \ref{Models}, the following probability measures denoted by $\mu_{\rho,s}$ are invariant.
\begin{enumerate}
    \item If $s=-1$, then $\mu_{\rho,-1}$ with $\rho \in [0,1]$ is distributed according to a product Binomial distribution, i.e., 
    $$\mu_{\rho,-1}(\eta) = 
        \prod_{v \in V} \binom{\alpha}{\eta_v}\rho^{\eta_v}(1-\rho)^{\alpha - \eta_v}.$$
    \item If $s=1$, then $\mu_{\rho, 1}$ with $\rho\in [0,1)$ is distributed according to a product Negative Binomial distribution, i.e.,
    $$ \mu_{\rho,1}(\eta) 
        =\prod_{v \in V} \frac{\Gamma\big(\alpha+\eta_v\big)}{\Gamma(\alpha)\cdot \eta_v!}\rho^{\eta_v}(1-\rho)^{\alpha}.$$
    \item If $s=0$, then $\mu_{\rho,0}$ with $\rho\geq0$ is distributed according to a product Poisson distribution, i.e.,
    $$\mu_{\rho,0}(\eta) = \prod_{v \in V} \cfrac{\rho^{\eta_v}}{\eta_v!}e^{-\rho}.$$
\end{enumerate}
\end{proposition}
\begin{proof}
The first two results are well-known and follow from the fact that the probability measures satisfy the detailed balance condition (see e.g. \cite{gibert}). 
For the third result,  a system of independent walkers on $V$ with single particle transition rates $p(v,w)$,  such that for all $w \in V$ 
\begin{equation*}
    \sum_{v\in V} \left( \rho(v)p(v,w) - \rho(w)p(w,v)\right) =0
\end{equation*}
has invariant product Poisson measures $\bigotimes_{v \in V} \text{Pois}(\rho(v))$  (see e.g. \cite{Derman}). Note that in our case we have  for all $(y,\sigma')\in V$,
\begin{equation}
    \sum_{(x,\sigma) \in V} p((x,\sigma),(y,\sigma')) = \sum_{(x,\sigma) \in V} p((y,\sigma'),(x,\sigma)) = \sum_{u\in\mathbb{Z}^d} \pi_{\sigma'}(u) + \lambda + \sum_{\sigma \in S} c(\sigma,\sigma').
\end{equation}
Therefore, the product Poisson measures with constant density $\rho>0$ are invariant.
\end{proof}

The following proposition provides the relation between the probability measures of Proposition \ref{invariant} and the duality functions of Theorem \ref{duality}.
\begin{proposition}\label{char}
Let $\mu \in \mathscr{P}(\Om_s)$, then $\mu = \mu_{\rho,s}$ if and only if for every $\xi \in \Om_{s,f}$ and every $v\in V$,
\begin{equation}\label{char eq}
    \int D_s(\xi,\eta)\dd\mu(\eta) = \left(\int D_s(\delta_v,\eta)\dd\mu(\eta)\right)^{|\xi|}.
\end{equation}
\end{proposition}
\begin{proof}
A straightforward calculation shows that \eqref{char eq} holds for $\mu_{\rho,s}$. The uniqueness property follows from the fact that $D_s(\xi, \eta)$ is a (multivariate) polynomial of order at most $|\xi|$. This implies that  \eqref{char eq} is actually a moment problem, which in the case of $\mu_{\rho,s}$ has a unique solution since the marginals have a finite  moment generating function (see e.g. \cite{kleiber}).
\end{proof}

 From this relation, the invariance of the probability measures also follows from the conservation of particles in the dual process. Namely, by duality and Fubini we have that 
 \begin{equation*}\begin{split}
     \int \E_{\eta}\left[D_s(\xi,\eta_t)\right]d\mu_{\rho,s}(\eta) 
     = \E_\xi \left[ \int D_s(\xi_t,\eta)d\mu_{\rho,s}(\eta)\right]
     = \left(\int D_s(\delta_v,\eta)\dd\mu_{\rho,s}(\eta)\right)^{|\xi|}.
 \end{split}\end{equation*}

\section{Ergodic theory of particle systems with homogeneous factorized duality polynomials}\label{3}

In this section we provide a characterization of the ergodic invariant probability measures satisfying a certain moment growth condition in a general
setting where we assume the existence of homogenous factorized duality polynomials, and the existence of successful coupling for the dual process.
This generalizes earlier results from \cite{ligg} chapter 8 for the symmetric exclusion process, and  \cite{kuoch} for the inclusion process. The characterization will be applied in Section \ref{4} to our models (see Theorem \ref{main theorem} below).

\subsection{Basic assumptions}
\subsubsection{Configurations}
We consider a configuration process $\{\eta(t):t \geq 0\}$ on (a subset of) the state space $\Omega=\mathbb{N}^G$, where
$G$ is assumed to be an infinite countable set. We denote by $S(t)$ the semigroup of this process, i.e.,
$S(t) f(\eta)= \E_\eta f(\eta(t))$.
We further denote by $\Omega_f$ the set of finite configurations, i.e., elements of
$\xi\in\Omega$ such that $|\xi|=\sum_x\xi_x<\infty$.

\subsubsection{Factorized duality functions}
We assume that there exists a duality function
\be\label{dualityf}
D:\Omega_f\times \Omega\to\R_+
\ee
such that
we have the duality relation
\be\label{dualityrel}
\E_\eta D(\xi, \eta(t))=\hE_\xi D(\xi(t), \eta).
\ee
We assume that $D(\emptyset, \eta)=1$ where $\emptyset$ denotes the empty configuration.
Moreover we assume that the duality functions are in homogeneous factorized form, i.e., of the form
\be\label{facto}
D(\xi, \eta)= \prod_{x\in G} d(\xi_x, \eta_x),
\ee
where $d(0,n)=1$, i.e., in the product only a finite number of factors is different from one, and where $d(k, \cdot)$ is a non-negative polynomial of
degree $k$. Moreover, we assume that every polynomial $p(n)$ of degree $k$ can be expressed as a linear combination of the polynomials
$d(r,n)$ with $0\leq r\leq k$.

We assume that both in the process $\{\eta(t):t \geq 0\}$ and in the dual process $\{\xi(t):t\geq 0\}$ the number of
particles is conserved.

In the examples of this paper, the duality functions are multivariate polynomials of degree $|\xi|$, and the dual process
$\{\xi(t): t\geq 0\}$ is either the same process (for the interacting examples) or the process obtained by reverting the velocities (the $\widehat{\text{RTP}}$ process defined in Theorem \ref{duality}).
In this section we take an abstract point of view and prove under general assumptions a structure theorem for the set of (tempered) invariant probability measures.

\subsubsection{Tempered probability measures}

Given a duality function, we define the $D$-transform of a probability measure $\mu$ on the configuration space $\Omega$
by
\be
\nonumber \widehat{\mu}(\xi)= \int D(\xi, \eta) \dd\mu (\eta),
\ee
where we implicitly assume that for all $\xi\in \Omega_f$, $D(\xi, \cdot)$ is $\mu$-integrable.
\bd\label{tempdef}
We then say that a probability measure $\mu$ is tempered if
\ben 
\item $\mu$ satisfies a uniform moment condition, i.e., for all $n\in\mathbb{N}$
\be\label{temp}
c_n:=\sup_{|\xi|\leq n} \int D(\xi, \eta) \dd\mu(\eta)<\infty.
\ee
\item $\mu$ is determined by  its $D$-transform, i.e., $\widehat{\mu}=\widehat{\nu}$ if and only if
$\mu=\nu$.
\item The following space
\[
\caD=\text{vct} \{ D(\xi, \cdot): \xi\in \Omega_f\},
\]
i.e., the vector space spanned by the functions $D(\xi,\cdot)$, is dense in $L^2(\mu)$.
\een
\ed
Notice that by the assumptions on the duality functions, the condition \eqref{temp} can be expressed equivalently by the requirement
that all moments of the occupation variables are finite uniformly in $x$, i.e., for all $n\in \mathbb{N}$
\[
\sup_{x\in G}\int \eta_x^n d\mu(\eta) <\infty.
\]
Using  H\"older's inequality,  we then also obtain that under \eqref{temp} we have that for all $n\in\N$
\be\label{cscond}
\sup_{\xi, \xi': |\xi|\leq n,  |\xi'|\leq n} \int D(\xi, \eta) D(\xi', \eta) \dd\mu(\eta)<\infty.
\ee

The condition that the $D$-transform determines the probability measure uniquely is implied by a growth condition
on $c_n$ which implies that the measure $\mu$ is uniquely determined by its multivariate moments. Examples of sufficient growth conditions  can be found in e.g. \cite[Section 3.2]{kleiber}.
In these settings, the condition
that $\caD$ is dense in $L^2(\mu)$ is also natural.
In the setting of processes of exclusion type, i.e., when there are at most $\alpha$ particles at each site, the condition of density of $\caD$ is natural and
follows from the Stone Weierstrass theorem, i.e., $\caD$ is uniformly dense in the set of continuous functions $\caC(\Omega)$.
\subsubsection{Assumptions on the dual process}\label{assumptions}
For the dual process $\{\xi(t):t\geq 0\}$, we assume that it is irreducible  on the sets
$\Omega_n=\{\xi: |\xi|=n\}$.
Moreover we assume that eventually the process $\{\xi(t):t\geq 0\}$ started at $\xi$ with $|\xi|=n$,  spreads out over the infinite set $\Omega_n$. This is
expressed via the condition
that for all $\xi'\in \Omega_f$
\be\label{spreadout}
\lim_{t\to\infty} \hP_{\xi} (\xi(t)\perp \xi')=1,
\ee
where we denote $\xi\perp \xi'$ the event that the supports of $\xi$ and $\xi'$ are disjoint.
In words, \eqref{spreadout} means that the probability that the configuration at time $t$ has non-zero occupation at
fixed sites tends to zero as $t\to\infty$.

\subsubsection{Ergodic probability measures}
We denote by $\caI$ the set of invariant probability measures of the process $\{\eta(t): t\geq 0\}$ and
by $\caT$ the set of tempered probability measures on $\Omega$.
Both $\caI$ and $\caT$ are convex sets.

We are then interested in characterizing
the ergodic probability measures which belong to $\caT$.
We recall that a probability measure $\mu\in\caI$ is ergodic if, for any $f\in L^2(\mu)$, $S(t) f= f$ for all $t\geq0$ implies $f=\int fd\mu$ almost surely.
The set of ergodic probability measures coincides with $\caI_e$,  the set of extreme points of $\caI$.
Ergodicity is implied by mixing (see e.g. \cite[Section 6.3]{walters}) which is the property that for all $f,g\in L^2(\mu)$
\be\label{mix}
\lim_{t\to\infty}cov_\mu (f, S(t)g)= \lim_{t\to\infty}\int \left(f-\int f \dd\mu\right)\left(S(t) g-\int g \dd\mu\right)\dd \mu=0.
\ee
By bilinearity of the covariance and the fact that, for $\mu\in \caI$, $S(t)$ is a contraction in $L^2(\mu)$, it suffices to show
\eqref{mix} for a set of functions  $f,g\in W$, where $W$ is such that the vectorspace spanned by $W$ is a dense subspace in $L^2(\mu)$.

\subsection{Successful coupling}
 We say that the dual process admits a \emph{successful coupling} if  for all $n\in\N$,  $\xi, \xi'\in \Omega_n$
there exists a coupling $ \{ (\xi^{(1)}(t), \xi^{(2)}(t)): t\geq 0\}$ of the processes $\{\xi(t):t\geq 0\}$ starting from $\xi$ and $\xi'$ 
such that the following stopping time
\[
\tau_{\xi, \xi'}=\inf \{ T>0: \xi^{(1)}(t)=\xi^{(2)}(t)\ \text{for all $t\geq T$}\}
\]
is a.s. finite. We call this stopping time the \emph{coupling time}. For this paper, we will make use of the following consequence of a successful coupling,
\be\label{succ}
\lim_{t\to\infty} \widehat{\pee}_{\xi, \xi'}(\xi^{(1)}(t)\not=\xi^{(2)}(t))=0,
\ee
where  $\widehat{\pee}_{\xi, \xi'}$ is the path space probability measure of $ \{ (\xi^{(1)}(t), \xi^{(2)}(t)): t\geq 0\}$ starting from $(\xi,\xi')$.

\subsection{Characterization of tempered invariant probability measures}
The following theorem has two parts: the first parts is well-known and appears in various context, e.g. \cite{ligg} chapter 2, chapter 8. We give its proof in this general context mainly for the sake of completeness. The second part is inspired by \cite{kuoch} in the context of the inclusion process.

\bt\label{coup}
\begin{enumerate}
\item
If there exists a succesful coupling for the dual process, then
for every tempered invariant probability measure $\mu$ there exists a function $f:\N\to [0, \infty)$ such that
for all $\xi\in \Omega_n$,
\be\label{bingo}
\widehat{\mu}(\xi)= f(n).
\ee
\item
If $\mu$ is a probability measure on $\Omega$ which is tempered, invariant and ergodic then $f(n)=f(1)^n$.
As a consequence, $\mu$ is a product measure.
\end{enumerate}
\et

\bpr
To prove item 1, by duality and the assumption that $\mu$ is tempered, we can use Fubini's theorem combined with duality to compute
\beq
\hE_\xi \widehat{\mu}(\xi(t))= \int \hE_\xi (D(\xi(t), \eta)) \dd\mu(\eta)= \int \E_\eta D(\xi, \eta(t))  \dd\mu(\eta)= \int D(\xi, \eta)  \dd\mu(\eta)=\widehat{\mu}(\xi).
\eeq
Here in the last equality we used the invariance of $\mu$. We conclude that $\widehat{\mu}$ is a harmonic function, which
is bounded on each $\Omega_n$ by the assumption that $\mu$ is tempered.

Therefore, using the assumed existence of a successful coupling,  by \eqref{succ} together with dominated convergence, we obtain for $\xi, \xi'\in \Omega_n$ the following
\beq
\widehat{\mu}(\xi)&=& \hE_{\xi, \xi'} \widehat{\mu}(\xi^{(1)}(t))
\nonumber\\
&=&
\hE_{\xi, \xi'} \widehat{\mu}(\xi^{(1)}(t))I(\xi^{(1)}(t))=\xi^{(2)}(t))) + o(1)
\nonumber\\
&=&
\hE_{\xi, \xi'} \widehat{\mu}(\xi^{(2)}(t))I(\xi^{(1)}(t))=\xi^{(2)}(t))) + o(1)
\nonumber\\
&=&
\hE_{\xi'} (\widehat{\mu}(\xi(t))) + o(1)
=\widehat{\mu}(\xi') +o(1),
\eeq
where $o(1)\to 0$ as $t\to\infty$.
This gives that $\widehat{\mu}$ is constant on $\Omega_n$, i.e., $\widehat{\mu} (\xi)= f(n)$ for some $f:\N\to\R$.

To prove item 2, we start by using the ergodicity combined with duality to write
\beq\label{bourako}
&&\int D(\xi, \eta) \dd\mu(\eta)\int D(\xi', \eta)\dd\mu(\eta)
\nonumber\\
&=&
\lim_{T\to\infty} \frac1T\int_0^T \int D(\xi, \eta) S(t) D(\xi', \cdot)(\eta) \dd\mu(\eta)\dd t
\nonumber\\
&=&
\lim_{T\to\infty} \frac1T\int_0^T \int D(\xi, \eta) \hE_{\xi'} D(\xi(t). \cdot)(\eta) \dd\mu(\eta)\dd t,
\eeq
Here in the first step we used that $\frac{1}{T}\int_0^T S(t)D(\xi',\cdot)(\eta)\dd t \to \int D(\xi',\eta)d\mu(\eta)$ holds almost surely and in $L^1(\mu)$, and in the second step we used duality.
Now let $\xi\in \Omega_n,\ \xi'\in \Omega_m$ be given. By item 1, we have $f(n)=\widehat{\mu}(\xi),\ f(m)=\widehat{\mu}(\xi')$.
By the homogeneous factorization of $D$, we have for $\xi\in\Omega_n,\ \xi'\in\Omega_m,\ \xi\perp\xi'$
\[
D(\xi, \eta) D(\xi', \eta)= D(\xi+\xi', \eta),
\]
and therefore, if $\xi \perp \xi'$, we have that
\[
\int D(\xi, \eta) D(\xi', \eta) \dd\mu(\eta)= f(n+m).
\]
Now combine \eqref{bourako} and  the assumption \eqref{spreadout} with the temperedness of the probability measure $\mu$ to
conclude
\beq
f(n)f(m)&=&
\int D(\xi, \eta) \dd\mu(\eta) \int D(\xi', \eta)\dd\mu(\eta)
\nonumber\\
&=&\frac1T\int_0^T \int D(\xi, \eta) \hE_{\xi'} D(\xi(t), \cdot)(\eta) \dd\mu(\eta)  + o(1)
\nonumber\\
&=&
\frac1T\int_0^T \int D(\xi, \eta) \hE_{\xi'} D(\xi(t), \cdot)(\eta)  I(\xi(t)\perp \xi)\dd\mu(\eta) +o(1)
\nonumber\\
&=&
\frac1T\int_0^T \int  \hE_{\xi'} D(\xi+\xi(t), \cdot)(\eta)  I(\xi(t)\perp \xi)\dd\mu(\eta) +o(1)
\nonumber\\
&=&
\frac1T\int_0^T \E_{\xi'}\left( I(\xi(t)\perp \xi) f(n+m)\right) + o(1)
\nonumber\\
&=&
f(n+m) + o(1),
\eeq
where $o(1)\to 0 $ as $T\to\infty$ via \eqref{cscond}, \eqref{temp}.
This proves that for all $\xi\in\Omega_n, \xi'\in \Omega_m$ we have
\[
\widehat{\mu}(\xi+\xi')= \widehat{\mu}(\xi)\widehat{\mu}(\xi'),
\]
which gives $f(n)= f(1)^n$.
This implies
that for all $x_1, \ldots, x_n\in G$, $k_1, \ldots, k_n\in \N$, 
\[
\int \prod_{i=1}^n d(k_i, \eta_{x_i}) \dd\mu(\eta)= f(1)^{k_1+\ldots+ k_n}= \prod_{i=1}^n \int d(k_i, \eta) \dd\mu(\eta),
\]
which implies that $\mu$ is a product measure.
\epr

In the next theorem we prove that invariant tempered product probability measures are ergodic. This, combined with Theorem \ref{coup}, completes the characterization of the set of
tempered ergodic probability measures.

We introduce
     $$K:=\left\{\int D(\delta_x, \eta)d\mu(\eta): \mu \ \text{is an invariant tempered product probability measure}\right\}.$$
\bt\label{ergomeas}
\begin{enumerate}
\item If $\mu$ is an invariant tempered product probability measure then it is ergodic.
\item
If there exists a successful coupling for the dual process,
then the only tempered invariant probability measures which are ergodic are the product probability measures $\mu_\theta$ for which
$\widehat{\mu}_\theta (\xi)= \theta^{|\xi|}$ with $\theta\in K$.
\item If there exists a successful coupling for the dual process,
then
\[
(\caT\cap \caI)_e= \caT\cap\caI_e= \{ \mu_\theta: \theta\in K\},
\]
where $(\caT\cap \caI)_e$ are the extreme points of $\caT\cap \caI$.
\end{enumerate}
\et
\bpr
For item 1, as indicated in the section where we defined mixing, it suffices to show that
\be\label{dmix}
\lim_{t\to\infty} \int D(\xi, \eta) \E_\eta D(\xi', \eta(t)) \dd\mu(\eta)=\widehat{\mu}(\xi)\widehat{\mu}(\xi')
\ee
because by assumption the vectorspace spanned by the $D(\xi, \cdot)$ is dense in $L^2(\mu)$.

Using duality, the assumption \eqref{spreadout}, the product character of the probability measure $\mu$ as well
as the assumed temperedness of $\mu$ (cf. \eqref{cscond}), and denoting $o(1)$ for a term which converges to zero as $t\to\infty$, we get
\beq
\int D(\xi, \eta) \E_\eta D(\xi', \eta(t)) \dd\mu(\eta)
\nonumber
 &=& \int D(\xi, \eta) \hE_{\xi'} D(\xi(t), \eta) \dd\mu(\eta)
\nonumber\\
 &=& \hE_{\xi'}\int D(\xi, \eta)  D(\xi(t), \eta) I(\xi(t)\perp \xi) \dd\mu(\eta) + o(1)
\nonumber\\
&=& \left(\int D(\xi, \eta) \dd\mu(\eta)\int  \hE_{\xi'}D(\xi(t), \eta) \dd\mu(\eta)  \right) + o(1)
\nonumber\\
&=&\int D(\xi, \eta) \dd\mu(\eta)\int \E_\eta(D(\xi', \eta(t))) \dd\mu(\eta) + o(1)
\nonumber\\
&=&
\int D(\xi, \eta) \dd\mu(\eta)\int D(\xi', \eta) \dd\mu(\eta) + o(1)
\nonumber\\
&=& \widehat{\mu}(\xi)\widehat{\mu}(\xi') + o(1).
\eeq
Item 2 follows immediately from item 1 and item 2 of Theorem \ref{coup}.
To prove item 3, we only have to prove that
\[
(\caT\cap\caI)_e= \caT\cap \caI_e.
\]
The implication ``$\mu\in \caT\cap \caI_e$ implies $\mu\in (\caT\cap\caI)_e$'' is obvious.
To prove the other implication, start from $\mu\in (\caT\cap\caI)_e$ and assume that we have
\[
\mu= \lambda \nu_1+(1-\lambda) \nu_2,
\]
with $\nu_1, \nu_2\in \caI$ and $0<\lambda<1$. Then we have, because $\mu\in\caT$, that $\nu_1, \nu_2\in \caT$, and therefore,
$\nu_1, \nu_2\in\caT\cap\caI$. But then, using that $\mu\in (\caT\cap\caI)_e$ we have
$\mu=\nu_1=\nu_2$, therefore we conclude that $\mu\in \caI_e$.
\epr

\section{Existence of a Successful Coupling}\label{4}
We can now state the main result of our paper, i.e., the characterization of the tempered ergodic probability measures for the three models of Section \ref{Models}.
\begin{theorem}\label{main theorem}
For all $s \in \{-1,0,1\}$, the probability measures $\mu_{\rho,s}$ defined in Proposition \ref{invariant} are the only tempered ergodic probability measures of the process generated by $\mathscr{L}_s$.
\end{theorem}
By Theorem \ref{ergomeas} we  need to show the the dual processes defined in Section \ref{dual processes} satisfy the assumptions from Section \ref{assumptions}, along with the existence of a successful coupling. We will start by proving that the original assumptions hold. 

The irreducibility of the processes on the sets $\Om_{s,n} = \{\xi \in \Om_{s,f} : \sum_{x}\xi_x=n\}$ is clear from the irreducibility of the single particle random walk. For \eqref{spreadout}, let $\xi'\in \Om_{s,n}$  and let $(v_i)_{i=1}^n\subset V$ be the coordinates of the particles in the configuration $\xi'$. Note then that for every $\xi \in \Om_{s,f}$
\begin{equation*}
    \widehat{\p}_\xi(\xi(t)\not\perp \xi') \leq \sum_{i=1}^n \widehat{\p}_\xi(\xi_{v_i}(t) \geq 1) \leq\sum_{i=1}^n \widehat{\E}_\xi[\xi_{v_i}(t)].
\end{equation*}
We are able to write $\xi_{v_i}(t) = \frac{1}{c_{\alpha,s}} D_s(\delta_{v_i},\xi(t))$ cf. Remark \ref{single particle duality}. Hence, using duality
\begin{equation*}
    \widehat{\E}_\xi[\xi_{v_i}(t)] =\E_{v_i}[\xi_{v(t)}] = \sum_{w\in V} p_t(v_i,w)\xi_w,
\end{equation*}
where $v(t)$ is the path of a particle under the dynamics of the original process starting from $v_i$, and $p_t(v, w)$ is the corresponding transition kernel. Here we also used that the dual of the dual is the original process (cf. Theorem \ref{duality}, item 3). Because $\xi$ is finite, the sum on the right-hand side is actually a finite sum,  and so \eqref{spreadout} follows if $p_t(v,w)\to 0$ as $t\to \infty$ for all $v,w \in V$. To see that this holds, note that for all $x,y,z\in\mathbb{Z}^d$ and $\sigma,\sigma'\in S$  we have that 
$$p_t((x,\sigma),(y,\sigma')) = p_t((x+z,\sigma),(y+z,\sigma')).$$ 
Therefore, there can not exist an invariant probability measure for the single particle random walk, which means that the random walk is either null-recurrent or transient. Hence we indeed have that $\lim_{t\to\infty} p_t(v,w)=0$ for all $v,w\in V$ (see e.g. \cite[p. 26]{Chung}).

In order to prove the existence of a successful coupling for the dual processes we proceed as follows:

First, we consider multi-layer symmetric independent random walkers (IRW)  on $V$ where the jump rates depend on the layer, i.e., $\widehat{\text{RTP}}$ with $\lambda = 0$ in \eqref{transition rtp}, which will be needed for the proof of our models. Second, we deal with interacting particles, distinguishing the transient and recurrent cases for both models. Finally, we prove the existence of a successful coupling for a general RTP, distinguishing the cases where there are random walk jumps and the case where there are only active jumps. 

In order to prove the successful coupling of finite configurations with identical particle numbers, we pass to a more convenient labeled particle configuration, i.e., when $\xi \in \mathbb{N}^V$ with $\sum_{v \in V} \xi_v = n$, then $\xi = \sum_{i=1}^n \delta_{(x_i,\sigma_i)}$ and we identify $\xi$ with $\big((x_1,\sigma_1), ..., (x_n,\sigma_n)\big)\in V^n$ where over the course of time, these initially chosen labels remain fixed. With this prescription, the configuration process $\xi(t)$ induces a unique process $\big((X_1(t),\sigma_1(t)), ..., (X_n(t),\sigma_n(t))\big)$ on $V^{n}$.

\subsection{Successful coupling of multi-layer symmetric IRW}
The proof of existence of a successful coupling of multi-layer symmetric IRW makes use of the Ornstein-coupling which is also used for the existence of a successful coupling of simple symmetric IRW on $\mathbb{Z}^d$. The argument can be found in e.g. \cite{Hollander}, however for completion we will give a proof here as well.
\begin{proposition}\label{coupling rw zd}
For all $n\in\mathbb{N}$ and  $\textbf{y}^{(1)},\textbf{y}^{(2)} \in \left(\mathbb{Z}^d\right)^n $, there exists a successful coupling $\big(\textbf{\emph{Y}}^{(1)}(t), \textbf{\emph{Y}}^{(2)}(t)\big)$ of simple symmetric IRW on $\mathbb{Z}^d$ with initial conditions $\textbf{\emph{Y}}^{(1)}(0) =\textbf{y}^{(1)}$ and $\textbf{\emph{Y}}^{(2)}(0) =\textbf{y}^{(2)}$.
\end{proposition}
\begin{proof}
 Since the particles move independently,  we only have to show that there exists a successful coupling of two simple symmetric random walkers on $\mathbb{Z}^d$. Namely, if we can successfully couple any two particles in the two configurations $\textbf{Y}^{(1)}(t) = \big( Y^{(1)}_1(t),...,Y^{(1)}_n(t)\big)$ and $\textbf{Y}^{(2)}(t)= \big( Y^{(2)}_1(t),...,Y^{(2)}_n(t)\big)$, then every stopping time 
 \begin{equation*}
     \tau_i :=\inf\left\{T>0: Y^{(1)}_i(t) = Y^{(2)}_i(t) \text{ for all $t\geq T$}\right\}
 \end{equation*}
 is a.s. finite. Note that the coupling time of $\textbf{Y}^{(1)}(t)$ and $\textbf{Y}^{(2)}(t)$ is then equal to $\tau = \max_{1\leq i\leq n}\tau_i$, which is therefore also a.s. finite. 
 
For the successful coupling of the pair $ Y_i^{(1)}(t)$ and $Y_i^{(2)}(t)$, let $\{e_1, e_2, ..., e_d\}$ be the standard basis vectors of $\mathbb{Z}^d$. Then we can write 
\begin{equation*}
    Y_i^{(1)}(t) - Y_i^{(2)}( t) = a_1(t)e_1 + a_2(t)e_2 + ... + a_d(t)e_d.
\end{equation*}
Here every $a_k(t)$ is a simple symmetric random walk on $\mathbb{Z}$. Now define the stopping times 
\begin{equation*}
    \tau_{a_k}:= \inf\{t\geq 0: a_k(t) = 0\}.
\end{equation*}
It is clear that every $\tau_{a_k}$ is a.s. finite. After time $\tau_{a_k}$ we let the processes $Y_i^{(1)}(t)$ and $Y_i^{(2)}(t)$ copy each others jumps in the direction of $e_k$, i.e. $a_k(t) = 0$ for all $t\geq \tau_{a_k}$. The proof is now finished after the observation that $\tau_i = \max_{1\leq k\leq d}\tau_{a_k}$.
\end{proof}

\begin{proposition}\label{coupling rw}
For all $n\in\mathbb{N}$ and  $\textbf{y}^{(1)},\textbf{y}^{(2)} \in V^n$, there exists a successful coupling $\big(\textbf{\emph{Y}}^{(1)}(t), \textbf{\emph{Y}}^{(2)}(t)\big)$ of multi-layer symmetric IRW on $V$, i.e., $\widehat{\text{RTP}}$ with $\lambda=0$, with initial conditions $\textbf{\emph{Y}}^{(1)}(0) =\textbf{y}^{(1)}$ and $\textbf{\emph{Y}}^{(2)}(0) =\textbf{y}^{(2)}$.
\end{proposition}
\bpr
 Similarly as in the proof of Proposition \ref{coupling rw zd},  we only have to show that there exists a successful coupling of two  random walkers $ (Y^{(1)}(t),\sigma^{(1)}(t))$ and $(Y^{(2)}(t),\sigma^{(2)}(t))$ on $V$. Initially we let the two random walkers evolve independently, up until the stopping time $\varsigma$ defined as
 \begin{equation}\label{internal stopping time}
    \varsigma := \inf{\{t\geq0: \sigma^{(1)}(t) = \sigma^{(2)}(t)\}}.
\end{equation}
Note that this stopping time is a.s. finite since the set $S$ is finite and the transition rates $\{c(\sigma,\sigma'): \sigma,\sigma' \in S\}$ on $S$ are irreducible. After the stopping time $\varsigma$, we let the two random walkers copy each others internal state jumps, i.e., we define the processes $(\tilde{Y}^{(i)}(t), \tilde\sigma^{(i)}(t))$ for $i=1,2$ such that  $(\tilde{Y}^{(i)}(t), \tilde\sigma^{(i)}(t)) = (Y^{(i)}(t), \sigma^{(i)}(t))$ for $t\leq \varsigma$ and $\tilde\sigma^{(i)}(t+\varsigma) = \tilde\sigma(t)$ for $t\geq0$, where $ \tilde\sigma(t)$ is an internal state process starting from $\tilde\sigma(0) = \sigma^{(1)}(\varsigma) = \sigma^{(2)}(\varsigma)$.

We can again write 
\begin{equation*}
    \tilde{Y}^{(1)}(\varsigma+t) - \tilde Y^{(2)}(\varsigma+ t) = a_1(t)e_1 + a_2(t)e_2 + ... + a_d(t)e_d,
\end{equation*}
where every $a_k(t)$ is a continuous-time nearest neighbor symmetric random walk on $\mathbb{Z}$ with (time-dependent) transition rates $2\pi_{\tilde\sigma(t)}(e_k)>0$. We again define the stopping times 
\begin{equation*}
    \tau_{a_k}:= \inf\{t\geq 0: a_k(t) = 0\},
\end{equation*}
and after time $\tau_{a_k}$ we let the processes $(\tilde{Y}^{(1)}(t), \tilde\sigma^{(1)}(t))$ and $(\tilde{Y}^{(2)}(t), \tilde\sigma^{(2)}(t))$ copy each others jumps in the direction of $e_k$. Note that the coupling time $\tau$ is now equal to $\tau = \varsigma +  \max_{1\leq k\leq d}\tau_{a_k}$, which is a.s. finite.
\epr

\subsection{Successful coupling of multi-layer SEP}
Let $\textbf{X}^{(1)}(t)$ and $\textbf{X}^{(2)}(t)$ be two finite configurations of multi-layer SEP particles with the same number of particles. We split the proof of the successful coupling up in two parts, namely the transient case and the recurrent case. 

\subsubsection{Transient case} Assume $d\geq 3$, then the random walk corresponding to the transition rates $\pi_\sigma(\cdot)$ is transient on $\mathbb{Z}^d$ for every $\sigma \in S$. Let $\textbf{Y}(t)$ be  an IRW process on $\mathbb{N}^V$ with finitely many particles. Since the transition rates are transient, for any $R>1$ and any starting position $\textbf{y}=(y_1, y_2, ..., y_n)$ such that $||y_i-y_j||_1>R$  for all $i\neq j$, with positive probability $p(R)$ the particles in $\textbf{Y}(t)$ starting from $\textbf{y}$ will never have \emph{collisions}. Here a collision means that there is a $t>0$ such that two particles $(Y_1(t),\sigma_1(t))$ and $(Y_2(t),\sigma_2(t))$ in the configuration $\textbf{Y}(t)$ are at neighboring positions of each other, i.e. we either have
$$||Y_1(t) - Y_2(t)||_1 = 1 \ \ \ \text{and}\ \ \ \sigma_1(t) = \sigma_2(t),$$
or
$$\hspace{7.5mm}Y_1(t) = Y_2(t) \ \ \  \ \ \ \ \ \ \ \ \ \ \text{and} \ \ \ c(\sigma_1(t),\sigma_2(t))>0.$$
It follows that, conditional on the event that there are no collisions, the multi-layer SEP particles move the same as multi-layer IRW particles. 

We now let the configurations $\textbf{X}^{(1)}(t)$ and $\textbf{X}^{(2)}(t)$ move for some time $T>0$. We denote by $R(T)$ the minimal distance between two particles in the same configuration at time $T$, i.e. $$R(T):=\min_{i\neq j}\{||X_i^{(1)}(t)-X_j^{(1)}(t)||_1, ||X^{(2)}_i(t) - X^{(2)}_j(t)||_1\},$$
with $X_i^{(1)}(t) \in \mathbb{Z}^d$ the position of particle $i$ in configuration $\textbf{X}^{(1)}(t)$. After time $T$, we start the coupling attempt by letting the SEP particles copy the jumps of IRW particles starting from $\textbf{X}^{(1)}(T)$ and $\textbf{X}^{(2)}(T)$. By Proposition \ref{coupling rw}, this attempt is successful with probability larger than $p(R(T))$. This proof is now finished by noting that in the transient case, we have that $R(T)\to \infty$ as $T\to\infty$ and $p(R) \to 1$ as $R\to\infty$.

\subsubsection{Recurrent case} For the case of $d\leq2$, where every $\pi_\sigma(\cdot)$ is recurrent, we will define the multi-layer SEP process on the  \emph{ladder graph} (see e.g. \cite{Giar}), i.e., we define the state space $\Om_{-1}' := \{0,1\}^{V\times A}$ with $A = \{1, 2,..., \alpha\}$). This space can be seen as the space  where on every site $v \in V$ there is a \emph{ladder} with $\alpha$ steps, and every particle chooses a step of this ladder if it moves to a new site. We can easily go back from a configuration $\eta' \in \Om_{-1}'$ to a configuration in $\eta \in \Om_{-1}$ by setting
\begin{equation}\label{ladder}
   \eta(x,\sigma) = \sum_{i=1}^\alpha \eta'(x,\sigma,i),\ \ \ \ \text{for all $(x,\sigma) \in V$}. 
\end{equation}

We now define the process on $\Om_{-1}'$ through the generator 
$$\mathscr{L}_{-1}'f(\eta') = \sum_{j,k=1}^\alpha \sum_{v,w\in V} \frac{p(v,w)}{\alpha} \eta'_{(v,j)}\left(\alpha - \eta'_{(w,k)}\right) \nabla_{(v,j),(w,k)}f(\eta'),$$
i.e., it is the simple symmetric exclusion process on $V\times A$ where particles choose a step on the ladder $A$ uniformly. It is easy to see that $\mathscr{L}_{-1}'$ on $\Om_{-1}'$  corresponds to the generator $\mathscr{L}_{-1}$ on $\Om_{-1}$  through $\eqref{ladder}$. The successful coupling of the multi-layer SEP  now follows from the successful coupling of the simple symmetric exclusion process on $\Om_{-1}'$. Since the set $V\times A$ is countable, this result is already known, for example in \cite[Chapter VIII]{ligg}.

\subsection{Successful coupling of multi-layer SIP}
The successful coupling of SIP on $\mathbb{Z}^d$ has already been shown by Kuoch and Redig in \cite{kuoch} and can be extended to our framework of multi-layer particles. In the transient case, this proof uses the same principle as the proof of a successful coupling for multi-layer \SEP\ above, i.e., we let the particles spread out far enough such that there are no collisions with positive probability, after which the particles move like independent random walkers for which there exists a successful coupling by Proposition \ref{coupling rw}. The proof of the recurrent case actually uses a similar approach as in the transient case, in which it lets the particles spread out over time and afterwards makes a coupling attempt. The probability that this coupling attempt is successful has non-zero probability. If the attempt fails, i.e., there is a collision, a new coupling attempt is made. Since these coupling attempts have non-zero probability of success and  are independent, there will be a successful coupling eventually. For more details on both proofs, see \cite{kuoch}.

\subsection{Successful coupling of \texorpdfstring{$\widehat{\textbf{RTP}}$}{TEXT}}
For the $\widehat{\text{RTP}}$ process we will also look at two cases, namely the case with random walk jumps of particles, i.e., $\kappa>0$ in \eqref{transition rtp}, and without random walk jumps. For the case of $\kappa>0$, we will see that the successful coupling of $\widehat{\text{RTP}}$ is a corollary of Proposition \ref{coupling rw} by copying the active and internal jumps of the process. If $\kappa=0$, then we will need  the additional assumption given in \eqref{irreducible internal}. With this assumption we are able to use a similar argument as in the proof of Proposition \ref{coupling rw} to prove the existence of a successful coupling.

\subsubsection*{Successful coupling of $\widehat{\text{RTP}}$ with $\kappa>0$}
For $(X(t),\sigma(t))$ a single $\widehat{\text{RTP}}$ particle, we can decouple the dynamics of $X(t)$ through the following decomposition,
\begin{equation*}\begin{split}\label{sum}
X(t) = Y(t) + Z(t),
\end{split}\end{equation*}
where $Y(t)$ is a symmetric random walk starting from $X(0)$ and $Z(t)$ are the active jumps starting from $0$, both of which are dependent of  $\sigma(t)$.

Since we are dealing with configurations of independent random walkers again, we only have to prove the existence of a successful coupling of two random walkers $(X^{(1)}(t),\sigma^{(1)}(t))$ and $(X^{(2)}(t),\sigma^{(2)}(t))$. Similarly as in the proof of Proposition \ref{coupling rw}, we let the two random walkers evolve independently up until the 
 stopping time $\varsigma$ defined as in \eqref{internal stopping time}. Afterwards, we let the random walkers copy each others internal state jumps, i.e., we define the processes $(\tilde{X}^{(i)}(t), \tilde{\sigma}^{(i)}(t))$ for $i=1,2$ as $(\tilde{X}^{(i)}(t), \tilde{\sigma}^{(i)}(t))=(X^{(i)}(t), \sigma^{(i)}(t))$ for $t\leq \varsigma$, and 
\begin{equation*}
    (\tilde{X}^{(i)}(\varsigma+t), \tilde{\sigma}^{(i)}(\varsigma+t)) = ( \tilde{Y}^{(i)}(t) + \tilde{Z}(t), \tilde{\sigma}(t)) 
\end{equation*}
where $\tilde{Y}^{(i)}(t)$ is again a symmetric random walk  starting from $X^{(i)}(\varsigma)$, $\tilde{Z}(t)$ are again the active jumps starting from 0, and $\tilde{\sigma}(t)$ is as defined in the proof of Proposition \ref{coupling rw}.  
Note that the difference of the positions of the two processes is now equal to
\begin{equation*}\begin{split}
\tilde{X}^{(1)}(\varsigma+t)& -\tilde{X}^{(2)}(\varsigma+t)=\tilde{Y}^{(1)}(t) -   \tilde{Y}^{(2)}(t),\\
\end{split}\end{equation*}
i.e., the difference between two symmetric random walkers.  The result now follows from Proposition \ref{coupling rw}.
 
\subsubsection*{Successful coupling of $\widehat{\text{RTP}}$ with $\kappa=0$ and $\lambda>0$}
Just as in the previous case, for two $\widehat{\text{RTP}}$ processes $(X^{(1)}(t),\sigma^{(1)}(t))$ and $(X^{(2)}(t),\sigma^{(2)}(t))$ we define the stopping time $\varsigma$ as in \eqref{internal stopping time}, and set up the processes $(\tilde{X}^{(i)}(t), \tilde{\sigma}^{(i)}(t))$ for $i=1,2$ such that  $(\tilde{X}^{(i)}(t), \tilde{\sigma}^{(i)}(t))=(X^{(i)}(t), \sigma^{(i)}(t))$ for $t\leq \varsigma$ and $\sigma^{(i)}(\varsigma +t) = \tilde\sigma(t)$ for $t>0$.

By \eqref{irreducible internal} we can now write 
\begin{equation}\label{pong}
    \tilde{X}^{(1)}(\varsigma + t) - \tilde{X}^{(2)}(\varsigma + t) = b_1(t)v(\sigma_1) + b_2(t)v(\sigma_2) + ... + b_m(t) v(\sigma_m)
\end{equation}
for $m=|S|$, $\sigma_k \in S$ and $b_k(t) \in \mathbb{Z}$ for all $k$ and $t\geq 0$. For every $k$, the  couple $(b_k(t), \tilde{\sigma}(t))$ is a random walk on $\mathbb{Z}\times S$ with the following dynamics: 
\begin{itemize}
    \item[-] If $\tilde{\sigma}(t) = \sigma_k$, $b_k(t)$ moves as a continuous-time nearest neighbor symmetric random walker on $\mathbb{Z}$ with rate $\lambda$. 
    \item[-] If $\tilde{\sigma}(t) \neq \sigma_k$, $b_k(t)$ does not move.
\end{itemize}
Since $S$ is finite and the transition rates $c(\sigma,\sigma')$ are irreducible, these random walks are recurrent as their discrete counterparts are recurrent. This implies that the stopping times
\begin{equation}\label{stopping rtp}
    \tau_{b_k}:= \inf\{t\geq 0: b_k(t) = 0\}
\end{equation}
are almost surely finite. After every time $\tau_{b_k}$, we let the process $\tilde{X}^{(2)}(\varsigma+t)$ copy the jumps of $\tilde{X}^{(1)}(\varsigma+t)$ in the direction of $v(\sigma_k)$. For the coupling time $\tau$, we then again have that $\tau = \varsigma + \max_{1\leq k\leq m} \tau_{b_k}$.

\begin{remark}
We are able to extend the results of this section to the case where we take $S$ countable. We would need the additional assumption that the transition rates $c(\sigma,\sigma')$ are positive recurrent, which ensures that we return to any $\sigma \in S$ in almost surely finite time. For the interacting particles we would then distinguish between the cases where the transition rates $p(v,w)$ in \eqref{generator} are transient and recurrent (note that the latter need not be the case where $d\leq 2$). For the $\widehat{\RTP}$, the positive recurrence of $c(\sigma,\sigma')$ ensures that every stopping time $\tau_{b_k}$ in \eqref{stopping rtp} is almost surely finite.
\end{remark}

\end{document}